\documentclass[table,a4paper,12pt]{amsart}
\usepackage{amssymb}
\usepackage{amscd}
\usepackage{hyperref}
\usepackage[backend=bibtex]{biblatex}
\usepackage{todonotes}
\usepackage{ytableau}
\usepackage{tikz}
\usetikzlibrary{cd}

\newtheorem{thm}{Theorem}[section]
\newtheorem{defn}[thm]{Definition}
\newtheorem{rem}[thm]{Remark}

\newtheorem{ex}[thm]{Example}
\newtheorem{prop}[thm]{Proposition}

\newtheorem{lemma}[thm]{Lemma}

\newcommand{\Dfn}[1]{\emph{\color{blue}#1}}

\newcommand{\os}[1]{[ #1 ]}
\DeclareMathOperator{\dom}{dom}
\DeclareMathOperator{\sort}{sort} 
\DeclareMathOperator{\End}{End}

\newcommand{\fS}{\mathfrak{S}}
\newcommand{\fC}{\mathfrak{C}}
\newcommand{\fB}{\mathfrak{B}}

\newcommand{\bQ}{\mathbb{Q}}
\newcommand{\bZ}{\mathbb{Z}}
\newcommand{\bN}{\mathbb{N}}

\newcommand{\SL}{\mathrm{SL}}
\newcommand{\Sp}{\mathrm{Sp}}
\newcommand{\GL}{\mathrm{GL}}

\newcommand{\cgd}{{cylindrical growth diagram}}
\newcommand{\wt}{\mathbf{wt}}
\newcommand{\rt}{\mathrm{rect}}
\title{Coboundary categories and local rules}
\author{Bruce W. Westbury}
\email{bruce.westbury@gmail.com}
\date{2017}

\begin{document}

\begin{abstract}
First we develop the theory of local rules for coboundary categories.
Then we describe the local rules in two main cases.
First for the quantum groups in general and in the seminormal
representations of the Hecke algebras. Then for crystals in general
and specifically for crystals of minuscule representations.
Finally we show how growth diagrams can be extended to construct
the action of the cactus group on highest weight words.
\end{abstract}
\maketitle
\tableofcontents

\section{Introduction}
Coboundary categories were originally defined by Drinfel$'$d
and were used to construct the braiding for quantum groups.
This approach fell into relative decline compared to the alternative
approach due to Jimbo. More recently, there has been a renewed
interest in coboundary categories since the combinatorial categories
of crystals are coboundary (but not braided).

The cactus groups are related to coboundary categories in the same
way as the braid groups are related to braided categories.
In particular, for any crystal $C$, the cactus group $\fC_r$
acts on the set of highest weight words in $\otimes^rC$
preserving the weight.

The basic example is to take $C$ to be the
crystal of the vector representation of $\GL(n)$. This gives an
action of $\fC_r$ on standard tableaux of size $r$ preserving the
shape. This action encompasses much of the standard combinatorics
of standard tableaux: each of the following operations on standard
tableaux is given by the action of a specific element of the
cactus group,
\begin{itemize}
	\item evacuation
	\item promotion
	\item dual Knuth moves
\end{itemize}
The coboundary structure on crystals, and other coboundary categories, gives a natural generalisation of these operations.

There are two constructions of the coboundary structure on
crystals. One uses the Lusztig involutions and the other the
Kashiwara involution on the crystal $B_\infty$. The construction
using the Lusztig involutions gives a coboundary structure on the
category of highest weight representations of the quantised
enveloping algebra of a finite dimensional semisimple Lie algebra.
The construction using the Kashiwara involution is more general
since it gives a coboundary structure on the
category of integrable representations of the quantised
enveloping algebra of a symmetrisable Kac-Moody algebra.

These two constructions are both technical. The aim of this paper
is to develop an effective construction of the action of the
cactus groups. This approach is based on the theory of local moves.
In the basic example of tableaux this approach recovers the
construction of promotion and evacuation using Fomin growth diagrams.
These were introduced as an effective construction of these operations
and have superseded the previous constructions using jeu-de-taquin.

The local rules for crystals of minuscule representations and their growth
diagrams are applied in \cite{Pfannerer2018}.

This paper is organised as follows. Section~\ref{sec:coboundary-categories}
gives the basic definitions of coboundary categories and cactus groups.
Then section~\ref{sec:local} follows \cite{Lenart2008} and develops the theory of local moves
in an arbitrary coboundary category. In section~\ref{sec:qg} we
describe the local rules for quantum groups using Drinfel$'$d
unitarisation, \cite{Drinfelprimed1989}, and we give an implicit construction of the
homomorphisms from cactus groups to Hecke algebras by giving the
local rules in the seminormal representations. In section~\ref{sec:crystal}
we give the local rules for the path models of crystals following \cite{Leeuwen1998}.
These local rules are particularly effective for crystals of
minuscule representations. Finally in section~\ref{sec:growth} we
show that constructions of operations on standard tableaux
using growth diagrams generalise to crystals.

\section{Coboundary categories}\label{sec:coboundary-categories}
Coboundary categories are monoidal categories with extra structure
which implies that taking the tensor product of two objects
in the two possible orders gives two objects which are naturally
isomorphic. We assume, for simplicity of exposition, that a monoidal
category means a strict monoidal category.

\subsection{Coboundary categories}
The original definition of a coboundary category from \cite{Drinfelprimed1989} is:
\begin{defn} A \Dfn{coboundary category} is a monoidal category together with natural maps $\sigma_{A,B}\colon A\otimes B \rightarrow B\otimes A$ for all objects $A,B$
	These maps are required to satisfy the three conditions
	\begin{itemize}
		\item $\sigma_{A,I}=1_A$ and $\sigma_{I,B}=1_B$
		\item $\sigma_{A,B}\circ \sigma_{B,A} = 1_{B\otimes A}$
		\item the following diagram commutes
		\begin{equation} \begin{CD} A\otimes B\otimes C @>{1_A\otimes\sigma_{B,C}}>>  A\otimes C\otimes B \\
		@V{\sigma_{A,B}\otimes 1_C}VV @VV{\sigma_{A,C\otimes B}}V \\
		B\otimes A\otimes C @>>{\sigma_{B\otimes A,C}}>  C\otimes B\otimes A
		\end{CD} \end{equation}
	\end{itemize}
\end{defn}

These conditions imply that the following diagram commutes
\begin{equation} \begin{CD} C\otimes B\otimes A @>{\sigma_{C\otimes B,A}}>>  A\otimes C\otimes B \\
@V{\sigma_{C,B\otimes A}}VV @VV{1_A\otimes\sigma_{C,B}}V \\
B\otimes A\otimes C @>>{\sigma_{B,A}\otimes 1_C}>  A\otimes B\otimes C
\end{CD} \end{equation}

\subsection{Cactus groups}
The finite presentations of the cactus groups were originally given
in~\cite{Devadoss1999}.
\begin{defn}\label{defn:cactus} The \Dfn{$r$-fruit cactus group}, $\fC_r$, has generators $s_{p,\,q}$ for $1\le p<q\le r$ and defining relations
\begin{itemize}
	\item $s_{p,\,q}^2=1$
	\item $s_{p,\,q}\, s_{k,\,l}=s_{k,\,l}\, s_{p,\,q}$ if $\os{p,q}\cap\os{k,l}=\emptyset$
	\item $s_{p,\,q}\, s_{k,\,l}=s_{p+q-l,\,p+q-k}\, s_{p,\,q}$ if $\os{k,l}\subseteq\os{p,q}$
\end{itemize}
\end{defn}

Let $\fS_r$ be the symmetric group on $r$ letters.
There is a homomorphism $\fC_r\rightarrow \fS_r$ defined by $s_{p,\,q}\mapsto \widehat{s}_{p,\,q}$ where $\widehat{s}_{p,\,q}$ is the permutation
\begin{equation}
\widehat{s}_{p,\,q}(i)=\begin{cases}
p+q-i & \text{if $p\le i\le q$} \\
i & \text{otherwise}
\end{cases}
\end{equation}

Note that $\fC_r$ is generated by $s_{1,q}$ for $2\le q\le r$, since
\begin{equation}\label{eqn:gen}
s_{p,\,q}=s_{1,\,q}s_{1,\,q-p}s_{1,\,q}
\end{equation}

A different set of generators are the elements $\sigma_{p,\,s,\,q}$ defined by
\begin{equation}\label{eqn:cob}
\end{equation}
Let $\widehat{\sigma}_{p,\,s,\,q}$ be the image of $\sigma_{p,\,s,\,q}$ under the homomorphism $\fC_r\rightarrow \fS_r$. Then $\widehat{\sigma}_{p,\,s,\,q}$ is the permutation
\begin{equation}
\widehat{\sigma}_{p,s,q}(i)=\begin{cases}
i & \text{if $1\le i<p$} \\
q-s+i & \text{if $p\le i\le s$} \\
p-s+i-1 & \text{if $s< i\le q$} \\
i & \text{if $q<i\le r$}
\end{cases}
\end{equation}

The following relations between the two sets of generators are given in \cite[Lemma~3]{Henriques2006}.
\begin{align}\label{eqn:rel}
	\sigma_{p,q,r}\, s_{k,\,l} &=
\begin{cases}
	s_{k+q-r,\,l+q-r}	\sigma_{p,\,q,\,r} &\text{if $\os{k:l}\subseteq\os{p:r}$} \\
	s_{k+p-r-1,\,l+p-r-1}	\sigma_{p,\,q,\,r} &\text{if $\os{k:l}\subseteq\os{r+1:q}$}
\end{cases}	\\
	s_{p,\,q} &= \sigma_{p,\,r,\,q}\, s_{p,\,r}\, s_{r+1,\,q}
\end{align}

Part of the relation between coboundary categories and the cactus groups
is the following two connections.
First, for each $r\ge 0$ and each object $A$ of a coboundary category, there is a
natural action of $\fC_r$ on $\otimes^rA$.
Secondly, the free coboundary category on the category with one
	morphism is the groupoid $\coprod_{r\ge 0} \fC_r$. This gives the
	generators in \eqref{eqn:cob}.

\section{Local rules}\label{sec:local}
First we define local rules in a coboundary category, following \cite{Lenart2008}.
\begin{defn}
	The \Dfn{local rules}, $\tau^A_{B,C}\colon A\otimes B\otimes C \rightarrow  A\otimes C\otimes B$, are defined to be the composite
	\begin{equation}
	A\otimes B\otimes C \xrightarrow{\sigma_{A,B}\otimes 1_C} B\otimes A\otimes C \xrightarrow{\sigma_{B,A\otimes C}} A\otimes C\otimes B
	\end{equation}
\end{defn}

The coboundary structure is recovered by taking
$\sigma_{B,C}=\tau^I_{B,C}$.
The coboundary condition can then be rewritten as
\begin{equation*}
\tau^C_{A,B}\, \sigma_{A,C}\otimes 1_B \, 1_A\otimes\sigma_{B,C}
= 
\sigma_{B,C}\otimes 1_A\, \tau^B_{A,C}\, \sigma_{A,B}\otimes 1_C
\end{equation*}
as maps $A\otimes B\otimes C\to C\otimes B\otimes A$.

In the rest of this section we give some further properties of local rules.

\begin{lemma}\label{lemma:rev} For all $A,B,C$, $\tau^A_{B,C}\,\tau^A_{C,B}=1_{A\otimes B\otimes C}$.
\end{lemma}
\begin{proof} The following diagram is commutative
\begin{equation}\begin{CD}
A\otimes B\otimes C @>{\sigma_{A,B}\otimes 1_C}>> B\otimes A\otimes C @>{\sigma_{B,A\otimes C}}>> A\otimes C\otimes B \\
@V{1_A\otimes\sigma_{B,C}}VV @V{\sigma_{B\otimes A,C}}VV @VV{\sigma_{A,B}\otimes 1_C}V \\
A\otimes C\otimes B @>>{\sigma_{A,C\otimes B}}> C\otimes B\otimes A @>{1_C\otimes\sigma_{B,A}}>> C\otimes A\otimes B \\
@V{1_A\otimes\sigma_{C,B}}VV @VV{\sigma_{C,B}\otimes 1_A}V @VV{\sigma_{B,A\otimes C}}V \\
A\otimes B\otimes C @>>{\sigma_{A,B\otimes C}}> B\otimes C\otimes A @>>{\sigma_{B\otimes C,A}}> A\otimes B\otimes C
\end{CD}\end{equation}
This completes the proof since the top edge is $\tau^A_{B,C}$, the right edge is $\tau^A_{C,B}$,
the left edge is $1_{A\otimes B\otimes C}$ and
the bottom edge is $1_{A\otimes B\otimes C}$.
\end{proof}

\begin{lemma}\label{lemma:addb} For all $A,B,C,D$,
	\begin{equation}
\tau^A_{B,C\otimes D}=\tau^{A\otimes C}_{B,D}\, \tau^A_{B,C}\otimes 1_D
	\end{equation}
\end{lemma}
\begin{proof}

\begin{multline*}
\tau^{A\otimes C}_{B,D}\, \tau^A_{B,C}\otimes 1_D \\
= \big(\sigma_{B,A\otimes C\otimes D}\,
\sigma_{A\otimes C,B}\otimes 1_D\big)\,
\big(\sigma_{B,A\otimes C}\otimes 1_D\,
\sigma_{A,B}\otimes 1_{C\otimes D}\big) \\
= \sigma_{B,A\otimes C\otimes D}\,
\big(\sigma_{A\otimes C,B}\otimes 1_D\,
\sigma_{B,A\otimes C}\otimes 1_D\big)\,
\sigma_{A,B}\otimes 1_{C\otimes D} \\
= \sigma_{B,A\otimes C\otimes D}\,
\sigma_{A,B}\otimes 1_{C\otimes D}
= \tau^A_{B,C\otimes D}
\end{multline*}
\end{proof}

\begin{lemma}\label{lemma:addt} For all $A,B,C,D$,
	\begin{equation}
\tau^A_{B\otimes C,D}= \tau^A_{B,D}\otimes 1_C\, \tau^{A\otimes B}_{C,D}
	\end{equation}
\end{lemma}
\begin{proof}
By Lemma~\ref{lemma:addb} we have
\begin{equation*}
\tau^A_{D,B\otimes C} = \tau^{A\otimes B}_{D,C}\, \tau^A_{D,B}\otimes 1_C
\end{equation*}
Taking the inverse of both sides using Lemma~\ref{lemma:rev} gives the result.
\end{proof}

The next Lemma constructs the commutors $\sigma_{B,\otimes^{r+1}B}$
for $r>0$ from the local rules $\tau^{\otimes^{r-1}B}_{B,B}$.

\begin{lemma}\label{lem:lm} For each object $B$ and $r\ge 0$,
\begin{equation*}
\sigma_{B,\otimes^{r+1}B} = \tau^{\otimes^rB}_{B,B}\,\tau^{\otimes^{r-1}B}_{B,B}\dotsb\tau^{\otimes B}_{B,B}\,\tau^{I}_{B,B}
\end{equation*}
\end{lemma}

\begin{proof} The proof is by induction on $r$.
The base of the induction is the case $r=0$ which is the observation that $\tau^I_{B,C}=\sigma_{B,C}$.

Substitute $A=I$, $C=\otimes^rB$, $D=B$ in Lemma~\ref{lemma:addb}, \ref{lemma:addt}.
 This gives
\begin{equation*}
\sigma_{B,\otimes^{r+1}B} = \tau^{\otimes^rB}_{B,B}\, \sigma_{B,\otimes^rB}\otimes 1_B
\end{equation*}
\end{proof}

\subsection{Relations}
The $BK$-group was introduced \cite{Kirillov1995} as the group
generated by the Bender-Knuth involutions on standard tableaux.
This was then shown to be a quotient of the cactus group in \cite{Chmutov2016}.

Consider the free group generated by $\tau_i$.
Define $q_i$ by
\begin{equation*}
q_i = \tau_1(\tau_2\tau_1)\dotsb (\tau_i\tau_{i-1}\dotsc \tau_1)
\end{equation*}

\begin{prop} For $r>0$, the group $\fC_r$ is generated by $\tau_i$ for $1\le i\le r-1$
and defining relations are
\begin{align*}
\tau_i^2 &= 1 \\
\tau_i\tau_j &= \tau_j\tau_i &&\text{for $|i-j|>1$}\\
(\tau_iq_{k-1}q_{k-j}q_{k-1})^2 &= 1 &&\text{for $i+1<j<k$}
\end{align*}
\end{prop}

The inverse isomorphisms are defined on the generators as follows.
In one direction
\begin{equation*}
s_{i,\,j} \mapsto q_{j-1}q_{j-i}q_{j-1}
\end{equation*}
and in the other direction by
\begin{equation*}
\tau_i \mapsto \begin{cases}
s_{1,\,2} & \text{if $i=1$} \\
s_{1,\,2}s_{1,\,3}s_{1,\,2} & \text{if $i=2$} \\
s_{1,\,i}s_{1,\,i+1}s_{1,\,i}s_{1,\,i-1} & \text{if $i>2$}
\end{cases}
\end{equation*}

It is shown in \cite{Chmutov2016} that the BK-groups are proper quotients of the
cactus groups. However there is no known presentation of the BK groups.
We conjecture that the $BK$-groups are the groups in \cite{Krammer2008a}. These are the quotients of the cactus group under the six term relations.

\section{Quantum groups}\label{sec:qg}
The aim of this section is to compute $\tau^{\otimes^n V}_{V,V}$
for $V$ a highest weight representation of a quantised enveloping
algebra.
This calculation is based on Drinfeld unitarisation
\begin{equation}\label{eq:unit}
\sigma_{V,W} = R_{V,W}(R_{W,V}R_{V,W})^{-1/2}
\end{equation}
where $R$ is the $R$-matrix which gives the braiding.

\begin{defn}
	The \Dfn{$r$-string braid group}, $\fB_r$, is generated by $t_i$ for $1\le i\le r-1$
	and the defining relations are
	\begin{align*}
	t_i\,t_{i+1}\,t_i&=t_{i+1}\,t_i\,t_{i+1} \\
	t_i\,t_j&=t_j\,t_i\qquad\text{for $|i-j|>1$}
	\end{align*}
\end{defn}

\begin{defn} For $1\leq n\le r$ the \Dfn{Jucys-Murphy element} $J_n\in\fB_r$
	is defined by
	\begin{equation*}
	J_n = (t_n\dotsc t_2t_1)(t_1t_2\dotsc t_n)
	\end{equation*}
\end{defn}

\begin{thm}\label{thm:qg} For $n>0$,
	\begin{equation*}
	\tau_n = J_{n-1}^{1/2}\,t_n\,J_n^{-1/2}
	\end{equation*}
\end{thm}

\begin{proof}
Put $W=\otimes^nV$. Then we have
\begin{equation*}
R_{V,W} = t_1t_2\dotsc t_n \text{ and } R_{V,W} = t_n\dotsc t_2t_1
\end{equation*}
Then, by Drinfeld unitarisation,
\begin{equation}\label{eq:ct}
\sigma_{1,\,1,\,n} = (t_1t_2\dotsb t_n)J_n^{-1/2}
\end{equation}
Then, by Lemma~\ref{lem:lm}, the local rule satisfies
\begin{equation*}
\sigma_{1,\,1,\,n} = \sigma_{1,\,1,\,n-1}\,\tau_n
\end{equation*}
Substituting from \eqref{eq:ct} gives
\begin{align*}
\tau_n &= \sigma_{1,\,1,\,n-1}^{-1}\,\sigma_{1,\,1,\,n} \\
&= J_{n-1}^{1/2}(t_1t_2\dotsb t_{n-1})^{-1}(t_1t_2\dotsb t_n)J_n^{-1/2} \\
&= J_{n-1}^{1/2}\,t_n\,J_n^{-1/2}
\end{align*}
\end{proof}

\subsection{Hecke algebras}
For each $r\ge 0$ there is a Hecke algebra $H_r(q)$. 
This is an algebra over the field of rational functions, $\bQ(q)$.

The quantum integers are the Laurent polynomials given by
\begin{equation*}
[n] = \frac{q^n-q^{-n}}{q-q^{-1}}
\end{equation*}
For example, $[2]=q+q^{-1}$.
\begin{defn}
	The Hecke algebra $H_r(q)$ is generated by $u_i$ for $1\le i\le r-1$
	and the defining relations are
	\begin{align*}
	u_i^2 &= -[2]\, u_i \\
	u_i\,u_{i+1}\,u_i-u_i&=u_{i+1}\,u_i\,u_{i+1}-u_{i+1} \\
	u_i\,u_j&=u_j\,u_i\qquad\text{for $|i-j|>1$}
	\end{align*}
\end{defn}

The algebra homomorphisms $\bQ(q)\,\fB_r\to H_r(q)$ are given by
\begin{equation*}
t_i^{\pm 1}\mapsto q^{\pm 1} + u_i	
\end{equation*}

Let $U_q(N)$ be the quantised enveloping algebra of type $A_N$.
Let $V$ be the vector representation; this is a fundamental
representation of dimension $N+1$. Then quantum Schur-Weyl
duality, \cite{MR841713}, is the statement that we have surjective homomorphisms
\begin{equation*}
H_r(q) \to \End_{U_q(N)}(\otimes^r V)
\end{equation*}
These are isomorphisms for $r\le \dim V$.

We also have homomorphisms $\fC_r\to \End_U(\otimes^r V)$
since $U_q(N)$-mod is a coboundary category. These homomorphisms
can be lifted to a homomorphism $\fC_r\to H_r(q)$ which is
independent of $N$.

\begin{lemma}\label{ex:VV}
	The cactus commutor $\sigma_{V,V}$ is the involution
	\begin{equation*}
	t_1(t_1^2)^{-1/2} = 1+\frac2{[2]}u_1	
	\end{equation*}
\end{lemma}
\begin{proof}
	We have
	\begin{equation*}
	t_1=q+u_1\qquad t_1^2=q^2+(q-q^{-1})u_1
	\end{equation*}
	The spectral decomposition of $t_1^2$ is
	\begin{equation*}
	t_1^2= q^2\left(1+\frac1{[2]}u_1\right)+q^{-2}\left(-\frac1{[2]}u_1 \right)
	\end{equation*}
	Hence
	\begin{equation*}
	(t_1^2)^{-1/2}= q^{-1}\left(1+\frac1{[2]}u_1\right)+q\left(-\frac1{[2]}u_1 \right)
	= q^{-1}-\left(\frac{q-q^{-1}}{q+q^{-1}} \right)u_1
	\end{equation*}
	Note that $t_1^{-1}=q^{-1}+u_1$, so $(t_1^2)^{-1/2}\ne t_1^{-1}$.
	
	Then multiplying by $t_1=q+u_1$ gives the result.
\end{proof}
This shows that the homomorphism $\bQ(q)\,\fC_r\to H_r(q)$ is surjective; answering
\cite[\S~10,~Question~2]{Kamnitzer2009a}.

\subsection{Seminormal form}
Young's seminormal forms are representations of the symmetric groups.
These were introduced in \cite[Theorem IV]{MR1576146}. Here we give
the analogous construction for the Hecke algebras following
\cite{MR1988991} and \cite{MR1976700}.

Fix a shape $\lambda$ of size $r$. Then we construct a 
representation of $H_r(q)$. The representation has basis the set
of standard tableaux of shape $\lambda$. For a standard tableau, $T$, let $s_iT$ be the
tableau obtained from $T$ by interchanging $i$ and $i+1$.
If $s_iT$ is not standard then put $s_iT=0$.

The content vector of a standard tableau $T$ of size $n$ is a function
$c_T\colon [1,2,\dotsc ,n]\to\bZ$. The entry
$c_T(k)\in\bZ$ is given by $c_T(k)=j-i$ if $k$ is in box $(i,j)$ in $T$.
The \Dfn{content vector} of $T$ is the sequence $[c_T(1),\dotsc ,c_T(r)]$.
The content vector of $T$ determines $T$. The \Dfn{axial distance}
is $a_T(i) = c_{T}(i+1)-c_{T}(i)$.

The following is \cite[Theorem~3.22]{MR1444315}.
\begin{defn}
	Then we define the action of $u_i$ by
	\begin{equation*} u_i\:T = 
	\begin{cases}
	-\frac{[a-1]}{[a]}\:T + s_iT &\text{if $c_T(i+1)>c_T(i)$}\\
	-\frac{[a-1]}{[a]}\:T + \frac{[a-1][a+1]}{[a]^2}\:s_iT &\text{if $c_T(i)>c_T(i+1)$}	
	\end{cases}
	\end{equation*}
where $a$ is the axial distance $a_T(i)$.
\end{defn}

If $T$ is standard and $s_iT$ is not standard then $a=\pm 1$.
Then $u_i\:T=-[2]T$ if $a=1$ and $u_i\:T=0$ if $a=-1$.
Assume $T$ and $s_iT$ are standard and $a>1$. Then on the subspace
with ordered basis $(T,s_iT)$.
\begin{equation*}
u_i =	\left[ \begin {array}{cc} -\frac {[a-1]}{ [a]  }&1
\\ \noalign{\medskip}{\frac {[a-1][a+1]  }{
		[a+1] ^{2}}}&-\frac {[a+1]}{ [a]  }\end {array} \right]
	\qquad
t_i=\left[ \begin {array}{cc} \frac {q^a}{ [a]  }&1
\\ \noalign{\medskip}{\frac {[a-1][a+1]  }{
		[a+1] ^{2}}}&\frac {-q^{-a}}{ [a]  }\end {array} \right]
\end{equation*}
Then 
The Jucys-Murphy elements are represented by diagonal matrices
in seminormal representations.

This is based on the result, see \cite{MR1427801}.
\begin{prop}\label{prop:jm}
	\begin{equation*}
	J_i\:T = q^{2c_i(T)}\:T
	\end{equation*}
\end{prop}

\begin{thm} The action of $\tau_i$ is given by
	\begin{equation*} \tau_iT = 
	\begin{cases}
	-\frac{1}{[a]}\:T + s_iT &\text{if $c_T(i+1)>c_T(i)$}\\
	-\frac{1}{[a]}\:T + \frac{[a-1][a+1]}{[a]^2}\:s_iT &\text{if $c_T(i)>c_T(i+1)$}	
	\end{cases}
	\end{equation*}
\end{thm}

Assume $T$ and $s_iT$ are standard and $a_T(i)>1$. Then on the subspace
with ordered basis $(T,s_iT)$.
\begin{equation*}
\tau_i=\left[ \begin {array}{cc} \frac {1}{ [a]  }&1
\\ \noalign{\medskip}{\frac {[a-1][a+1]  }{
		[a+1] ^{2}}}&-\frac {1}{ [a]  }\end {array} \right]
\end{equation*}

\begin{proof}
It is sufficient to check the relation in Theorem~\ref{thm:qg}.

By Proposition~\ref{prop:jm},
	\begin{equation*}
	J_i^{\pm 1/2}\:T = q^{\pm c_i(T)}\:T
	\end{equation*}
	
	All the matrices are simultaneously block diagonal.
	Therefore it is sufficient to check this on each block.
	The check for the $2\times 2$ blocks is the identity that, for $r+s=a$,
	\begin{equation*}
	\left[ \begin {array}{cc} q^r&0
	\\ \noalign{\medskip}0&q^{-s}\end {array} \right] 	
	\left[ \begin {array}{cc} \frac {1}{ [a]  }&1
	\\ \noalign{\medskip}{\frac {[a-1][a+1]  }{
			[a+1] ^{2}}}&-\frac {1}{ [a]  }\end {array} \right]
	=
	\left[ \begin {array}{cc} \frac {q^a}{ [a]  }&1
	\\ \noalign{\medskip}{\frac {[a-1][a+1]  }{
			[a+1] ^{2}}}&\frac {-q^{-a}}{ [a]  }\end {array} \right]
	\left[ \begin {array}{cc} q^{-s}&0
	\\ \noalign{\medskip}0&q^{r}\end {array} \right] 	
	\end{equation*}
\end{proof}

\section{Crystals}\label{sec:crystal}
For each finite type Cartan matrix there is a monoidal category of finite crystals. These were shown to be coboundary categories
in~\cite{Henriques2006}. In this section we describe the local rules.
\subsection{Crystals}
In this section we give a summary of the basic theory of crystals.
This is based on \cite[Chapter~5]{Joseph1995}.

If $X$ is a set then $X_\ast$ is the pointed set $X\coprod \{0\}$.
A partial function $X\nrightarrow Y$ can be identified with a map of pointed sets $X_\ast\to Y_\ast$.

The minimal data for a \Dfn{normal crystal} is a finite set $B$ together with partial functions $e_i\colon B\nrightarrow B$ for $i\in I$.
Each $e_i$ is injective and nilpotent.

A \Dfn{morphism} $B\to B'$ is a partial function
$F\colon B_\ast\nrightarrow B'_\ast$ such that $F\circ e_i=e'_i \circ F$
for $i\in I$.

The minimal data for a crystal is usually presented as a directed graph with vertex set $B$
and edges labelled by $I$. There is an $i$-edge from $x\to y$
if and only if $e_i\, x = y$. The directed graph of a crystal has no oriented cycles.

The minimal data we have presented determines additional data.
In most presentations of the theory of crystals this data is 
included in the definition.

The first additional data we add are the partial functions $f_i\colon B\nrightarrow B$ for $i\in I$.
These are defined by $f_i\,x=y$ if and only if
$e_i\,y=x$. These are also injective and nilpotent.

Now we define functions $\varepsilon_i,\varphi_i\colon B\to\bN$ by
\begin{align*}
	\varepsilon_i(x) &= \max \{k|e_i^k\ne 0\} \\
	\varphi_i(x) &= \max \{k|f_i^k\ne 0\}
\end{align*}

The weight function is given by
\begin{equation*}
	\wt_i(x) = \langle \varepsilon(x)-\varphi(x) , \alpha^\vee_i \rangle
\end{equation*}

An element $x\in B$ is \Dfn{highest weight} if $e_i(x)=0$, or, equivalently, if $\varepsilon_i(x)=0$ for all $i\in I$.

A homomorphism between crystals induces a weight-preserving partial
function between the highest weight elements and is determined by this partial function.

The \Dfn{tensor product} of crystals $B$ and $C$ is constructed as
follows:

The set underlying $B\otimes C$ is $B\times C$. The partial functions,
$e_i$, for $i\in I$, are defined by
\begin{align*}
	e_i(x\otimes y) &= \begin{cases}
		e_i(x)\otimes y & \text{if $\varphi(x)\ge\varepsilon(y)$} \\
		x\otimes e_i(y) & \text{otherwise}
	\end{cases} \\
\end{align*}

The functions $\varepsilon_i$ and $\varphi_i$ are defined by
\begin{align*}
\varepsilon_i(x\otimes y) &= \varepsilon_i(x) +
\max\{0,\varepsilon_i(y)-\varphi_i(x)\} \\
\varphi_i(x\otimes y) &= \varphi_i(x) +
\max\{0,\varphi_i(x)-\varepsilon_i(y)\}
\end{align*}

Then we have $\wt(x\otimes y)=\wt(x)+\wt(y)$.

The highest weight elements of $B\otimes C$ are the elements $x\otimes y$
such that $x\in B$ is highest weight and
$\varepsilon_i(y)\le\varphi_i(x)$ for $i\in I$.

The main result of~\cite{Henriques2006} is that the monoidal category of crystals is a coboundary category. This gives
a natural action of $\fC_r$ on the set $B(\omega)$ for each dominant weight, $\omega$.

There is a second construction of the coboundary structure in \cite{Kamnitzer2009} using the Kashiwara involution. This construction is given on highest weight
elements by the formula
\begin{equation*}
\sigma_{B_\lambda,B_\mu}(b_\lambda\otimes c) = b_\mu \otimes \ast c
\end{equation*}
\subsection{Words}
Let $C$ be a crystal of a finite type Cartan matrix.
For $r\ge 0$, let $\otimes^r C$ be the crystal of words in $C$ of length $r$. Then we have a canonical isomorphism of crystals
\begin{equation}\label{eq:cdec}
\otimes^r C \cong \coprod_\omega C(\omega) \times B(\omega)
\end{equation}
where $B(\omega)$ is a set and $C(\omega)$ is a connected crystal.
This is an isomorphism of crystals so for $w\leftrightarrow (P,Q)$
we have $e_i w\leftrightarrow (e_i P,Q)$ and $f_i w\leftrightarrow (f_i P,Q)$ for all $\alpha\in I$.

Let $w\leftrightarrow (P,Q)$ under this correspondence. Then $P$ generalises the insertion tableau and $Q$ generalises the recording tableau. For $w\leftrightarrow (P,Q)$ and $w'\leftrightarrow (P',Q')$; then $P=P'$ means $w$ and $w'$ are in the same position in isomorphic components; and $Q=Q'$ means $w$ and $w'$ are in the same component.

In special cases there is a combinatorial construction of the sets $B(\omega)$ and a corresponding insertion algorithm. In general we can take $B(\omega)$ to be the set of highest weight words of length $r$ and weight $\omega$.

\subsection{Tableaux}
Take $C$ to be the crystal of the vector representation of $\GL(n)$.
Then the decomposition \eqref{eq:cdec} is the Robinson-Schensted correspondence with $B(\omega)$ the set of 
standard tableaux of shape $\omega$.

Let $\omega$ be a partition of size $r$. Then we have an action of
the $r$-fruit cactus group, $\fC_r$, on the set of standard tableaux
of shape $\omega$. The group $\fC_r$ is generated by the elements
$s_{1,\,p}$ for $2\le p\le r$; so the action of $\fC_r$ is determined
by the action of these elements.
For $2\le p\le r$, the action of $s_{1,\,p}$ is given by applying evacuation to the subtableau with entries $1,\dotsc ,p$ leaving the remaining entries fixed. 

\begin{ex} Take $C$ to be the crystal of the two dimensional representation of $\SL(2)$. The set $B(0)$ is the set of noncrossing perfect matchings on $r$ points. The action of $s_{1,\,p}$ is given
by the rule that each pair $(i,j)$ with $i<j$ gives a pair
\begin{equation*}
	\begin{cases}
(p-j+1,p-i+1) & \text{if $i<j\le p$}	\\
(p-i+1,j) & \text{if $i\le p<j$}	\\
(i,j) & \text{if $p<i<j$}
	\end{cases}
\end{equation*}
The case $r=6$ is shown in Figure~\ref{fig:cat}.
\end{ex}

Then we have the following properties:
 \begin{itemize}
 	\item evacuation is given by the action of $s_{1,\,r}$
 	\item promotion is given by the action of $s_{1,\,r}\,s_{2,\,r}$
 	\item the elements $s_{p,\,p+1}$ act trivially
 	\item the dual Knuth move $D_i$ is given by the action of $s_{i,\,i+2}$
 \end{itemize}
 
\begin{figure}
	\begin{tikzpicture}[shorten >=1pt,auto,node distance=3.5cm,semithick,state/.style={}]
	\node[state] (A) {\begin{ytableau} 1 & 2 & 3\\ 4 & 5 & 6\end{ytableau}};
	\node[state] (B) [right of=A] {\begin{ytableau} 1 & 2 & 4\\ 3 & 5 & 6 \end{ytableau}};
	\node[state] (C) [right of=B] {\begin{ytableau} 1 & 3 & 4\\ 2 & 5 & 6 \end{ytableau}};
	\node[state] (D) [below of=A] {\begin{ytableau} 1 & 3 & 5 \\ 2 & 4 & 6 \end{ytableau}};
	\node[state] (E) [right of=D] {\begin{ytableau} 1 & 2 & 5 \\ 3 & 4 & 6 \end{ytableau}};
	\path (A) edge [bend left=30] node {$(2,4)$} (E);
	\path (B) edge [bend right=30] node {$(3,5)$} (A);
	\path (B) edge node {$(1,3)$} (C);
	\path (C) edge node {$(3,5)$} (D);
	\path (D) edge [bend left=30] node {$(2,4)$} (E);
	\path (E) edge [bend left=30] node {$(1,3)$} (D);
	\path (A) edge [bend left=45] node {$(1,4)$} (C);
	\path (C) edge [bend left=45] node {$(2,5)$} (E);
	\path (C) edge [bend right=0] node {$(1,6)$} (E);
	\path (E) edge [bend left=30] node {$(1,5)$} (A);
	\path (D) edge [bend left=30] node {$(1,5)$} (B);
	\end{tikzpicture}
\caption{Representation of $\fC_6$}\label{fig:cat}
\end{figure}

\subsection{Local rules}
Fix a crystal $C$. Then the action of $\fC_r$ on $\otimes^rC$ is determined by the involutions $\tau^B_{C,C}$
where $B$ is arbitrary. These involutions are crystal homomorphisms and so are determined by the restriction
to highest weight elements.

\begin{defn} A representation is \Dfn{minuscule} if the Weyl group acts transitively on the weights of the representation.
\end{defn}

The non-trivial minuscule representations are known and,
for the convenience of the reader, we give the list here:
\begin{description}
	\item[type $A_n$] All exterior powers of the vector representation.
	\item[type $B_n$] The spin representation.
	\item[type $C_n$] The vector representation.
	\item[type $D_n$] The vector representation and the two half-spin representations.
	\item[type $E_6$] The two fundamental representations of dimension 27.
	\item[type $E_7$] The fundamental representation of dimension 56.
\end{description}
There are no nontrivial minuscule representations in types $G_2$, $F_4$ or $E_8$.

\begin{defn}\label{def:dom}
For each weight $\lambda$, there is a unique weight which is both
dominant and in the Weyl group orbit of $\lambda$. Denote this
element by $\mathrm{dom}_W(\lambda)$.
\end{defn}

\begin{ex} For $\GL(n)$, the corner labels are weakly decreasing sequences of integers of length $n$. The Weyl group is $\fS_n$ and $\dom_{\fS_n}$ takes a sequence of integers of length $n$ and rearranges into weakly decreasing order. In this case $\dom_{\fS_n}$
is denoted by $\sort$.
\end{ex}

\begin{ex} For $\Sp(2n)$, the corner labels are partitions of length $n$. The Weyl group is a hyperoctahedral group. The map $\dom_W$ takes a sequence of integers of length $n$, forms the absolute values, and rearranges into weakly decreasing order.
\end{ex}

The following interpretation of~\cite[Rule~4.1.1]{Leeuwen1998} is given in~\cite[Proposition~4.1]{Lenart2008}.
\begin{thm}\label{thm:dom} Using the notation in Figure~\ref{fig:localrule}; for minuscule crystals $B$ and $C$ the following are equivalent
	\begin{align*}
	\tau^A_{B,C}(b,h,v)&=(a,v',h') & \mu &= \mathrm{dom}_W (\kappa+\nu-\lambda) \\
	\tau^A_{C,B}(b,v',h')&=(a,h,v) & \lambda &= \mathrm{dom}_W (\kappa+\nu-\mu)
	\end{align*}
\end{thm}

\subsection{Bender-Knuth}
In this section we relate the Bender-Knuth involutions, introduced in \cite{Bender1972},
to local rules.

Denote the $i$-th exterior power of the vector representation of $\GL(n)$ by $\Lambda^i$. A tableau is \Dfn{semistandard}
if the entries weakly increasing along the
rows and strictly decrease down the columns. A tableau is \Dfn{dual semistandard}
if the entries strictly increasing along the
rows and weakly decrease down the columns.

It follows from the dual Pieri rule that the set of highest weight elements in the crystal of the tensor product
\begin{equation*}
\Lambda^{i_1}\otimes \Lambda^{i_2}\otimes \dotsb \otimes \Lambda^{i_r}
\end{equation*}
of weight $\lambda$ corresponds to dual semistandard tableaux of shape $\lambda$ with entries in $[r]$, weight $(i_1,i_2,\dotsc ,i_r)$. Hence the coboundary structure
on the category of crystals gives an action of the $r$-fruit cactus group, $\fC_r$, on the 
set of dual semistandard tableaux with entries in $[r]$ and shape $\lambda$.

Semistandard tableaux correspond to Gelfand-Tsetlin patterns.
A \Dfn{Gelfand-Tsetlin pattern} of length $r$ is a sequence of partitions,
\begin{equation*}
\emptyset = \lambda^{(0)} \subseteq \lambda^{(1)} \subseteq  \dotsb \subseteq \lambda^{(r)}
\end{equation*}
such that each skew shape $\lambda^{(k)}/\lambda^{(k-1)}$ is a horizontal strip.
The length is $r$ and the shape is the final shape $\lambda^{(r)}$.

Similarly, dual semistandard tableaux correspond to sequences of partitions,
\begin{equation}\label{eq:seq}
\emptyset = \lambda^{(0)} \subseteq \lambda^{(1)} \subseteq  \dotsb \subseteq \lambda^{(r)}
\end{equation}
such that each skew shape $\lambda^{(k)}/\lambda^{(k-1)}$ is a vertical strip.
The length is $r$ and the shape is the final shape $\lambda^{(r)}$.

The exterior powers of the vector representation of $\GL(n)$ are all minuscule.
Hence the local moves are given by Theorem~\ref{thm:dom}.
The local move, $\tau_i$, acting on the sequence \eqref{eq:seq}, changes $\lambda^{(i)}$ to
\begin{equation}
\sort(\lambda^{(i-1)}-\lambda^{(i)}+\lambda^{(i+1)})
\end{equation} and does not change the remaining
partitions. These local rules determine the action of $\fC_r$ on dual semistandard
tableaux with entries in $[r]$.

Denote the conjugate of a tableau $T$ by $T^\mathtt{t}$. The Bender-Knuth involutions
are involutions on semistandard tableaux. The Bender-Knuth involution $b_i$ acts by changing some of the entries $i$ of the tableau to $i+1$, and some of the entries $i+1$ to $i$, in such a way that the numbers of elements with values $i$ or $i+1$ are exchanged. 
The relation between the local rules $\tau_i$ acting on dual semistandard tableaux
and the Bender-Knuth involution acting on standard tableaux is:
\begin{lemma}For any dual semistandard tableau, $T$,
\begin{equation*}
b_i(T^\mathtt{t}) = \tau_i(T)^\mathtt{t}
\end{equation*}
\end{lemma}

\begin{ex} We apply $b_2$ to the tableau
	\begin{equation*} T = 
	\begin{ytableau}
	1 & 1 & 1 & 2 \\ 2 & 3 \\ 4
	\end{ytableau}
	\end{equation*}

The Gelfand-Tsetlin pattern is the sequence of partitions
\begin{equation*}
[], [3], [4,1], [4,2], [4,2,1], [4,2,1], \dotsc
\end{equation*}
Taking conjugates gives the sequence of partitions
\begin{equation*}
[],[1,1,1],[2,1,1,1],[2,2,1,1],[3,2,1,1],[3,2,1,1],\dotsc
\end{equation*}
Now apply the local rule
\begin{equation*}
\sort([1,1,1,0]-[2,1,1,1]+[2,2,1,1])=\sort([1,2,1,0])=[2,1,1]
\end{equation*}
Replace the partition $[2,1,1,1]$ by $[2,1,1]$ to get the sequence
\begin{equation*}
[],[1,1,1],[2,1,1],[2,2,1,1],[3,2,1,1],[3,2,1,1],\dotsc
\end{equation*}
Take conjugates to get the sequence
\begin{equation*}
[], [3], [3,1], [4,2], [4,2,1], [4,2,1], \dotsc
\end{equation*}
This is the Gelfand-Tsetlin pattern of the tableau
	\begin{equation*} b_2(T) = 
\begin{ytableau}
1 & 1 & 1 & 3 \\ 2 & 3 \\ 4
\end{ytableau}
\end{equation*}
\end{ex}
\subsection{Continuous crystals}
The theory of continuous crystals for any Coxeter group is developed
in \cite{Biane2009}. This is an extension of the Littelmann path model.
In particular they construct an analogue of the Lusztig involution
in \S 4.10. This then gives a coboundary structure on the category
of continuous crystals by the construction in \cite{Henriques2006}.

The local moves for these coboundary categories are given in
\cite[\S~5]{Leeuwen1998}.

The following is \cite[5.1~Rule]{Leeuwen1998}.
\begin{defn}
	Let $f\colon [0,1]\times [0,1]\to X_\bQ$ be a piecewise linear function.
	The function $f$ satisfies the \Dfn{local rules} if, for every
	pair of intervals $[s_0,s_1],[t_0,t_1]\subseteq [0,1]$ such that
	$f$ is linear on $\{s_0\}\times [t_0,t_1]$ and $[s_0,s_1]\times \{t_1\}$
	one has
	\begin{equation*}
		f(s,t)=\dom_W(f(s_0,t)+f(s,t_1)-f(s_0,t_1))
	\end{equation*}
\end{defn}

The following is \cite[5.3~Theorem]{Leeuwen1998}.
\begin{thm}
	Let $\kappa$, $\lambda$, $\nu$ be dominant weights.
	Let $\pi'$ be a $\kappa$-dominant path with $\pi'(1)=\lambda-\kappa$
	and $p$ a $\lambda$-dominant path with $p(1)=\nu-\lambda$.
	Then there is a unique $f\colon [0,1]\times [0,1]\to X_\bQ$
	which satisfies the local rules and such that $f(0,t)=\kappa+\pi'(t)$
	for $0\le t\le 1$ and $f(s,1)=\lambda+p(s)$ for $0\le s\le 1$.
\end{thm}

Putting $\mu=f(1,0)$ we define a $\kappa$-dominant path, $\pi$,
with $\pi(1)=\mu-\kappa$ by $\pi(s)=f(s,0)-\kappa$ for $0\le s\le 1$
and a $\mu$-dominant path, $p'$, with $p'(1)=\nu-\mu$ by
$p'(t)=f(1,t)-\mu$ for $0\le t\le 1$.

\begin{thm}
	Let $\kappa$, $\lambda$, $\nu$ be dominant weights.
	Let $\pi'$ be a $\kappa$-dominant path with $\pi'(1)=\lambda-\kappa$
	and $p$ a $\lambda$-dominant path with $p(1)=\nu-\lambda$. Then	
	\begin{align*}
	\tau^{C(\kappa)}_{{C(\lambda)},C(\mu)}(a,\pi',p)&=(a,\pi,p')  \\
	\tau^{C(\kappa)}_{C(\mu),C(\lambda)}(a,\pi,p')&=(a,\pi',p)
\end{align*}
where $a\in C(\kappa)$ is highest weight.
\end{thm}

\section{Growth diagrams}\label{sec:growth}
The local rules are used to build growth diagrams.
Growth diagrams are used to define operations.
The main examples are:
\begin{itemize}
	\item promotion; two row growth diagrams
	\item evacuation; triangular growth diagrams
	\item rectification; rectangular growth diagrams
\end{itemize}

Our notation for a highest weight element
$x\otimes y\in B\otimes C$ is
$\lambda	\xrightarrow{\:y\:}\mu$
where $\lambda=\wt(x)$ and $\mu=\wt(x\otimes y)$. 
This notation extends to highest weight words. Given a word
$w=x_1\otimes x_2\otimes \dotsb \otimes x_r$ define the weights
$\lambda_k$ for $0\le k\le r$ by $\lambda_0=0$ and
$\lambda_i=\lambda_{i-1}+\wt(x_i)$ for $1\le k\le r$.
Then the word $w$ is highest weight if and only if
$\lambda_{k-1}\le \varphi_i(x_k)$ for $i\in I$
and $1\le k\le r$. This implies that $\lambda_k$ is dominant
for $0\le k\le r$. The word $w$ is then represented by
\begin{equation}\label{eqn:word}
0\xrightarrow{\:x_1\:}\lambda_1		\xrightarrow{\:x_2\:}
\dotsb 	\xrightarrow{\,x_{r-1}\,}\lambda_{r-1}		\xrightarrow{\:x_r\:}\lambda_r
\end{equation}
If $C$ is minuscule then we can safely omit the edge labels in
\eqref{eqn:word} and Figure~\ref{fig:localrule} as these are determined by the corner labels.

\begin{defn} A \Dfn{cell} is a square as shown in
Figure~\ref{fig:localrule}. The directed edges represent highest
weight words using the notation in \eqref{eqn:word}.
\end{defn}
\begin{figure}
	\begin{center}
		\begin{tikzpicture}
		\matrix (m) [matrix of math nodes,row sep=3em,column sep=4em,minimum width=2em]
		{
			\lambda & \nu \\
			\kappa & \mu \\};
		\path[-stealth]
		(m-2-1) edge [->] node [left] {$h$} (m-1-1);
		\path[-stealth]
		(m-1-1) edge [->] node [above] {$v$} (m-1-2);
		\path[-stealth]
		(m-2-1) edge [->] node [below] {$v'$} (m-2-2);
		\path[-stealth]
		(m-2-2) edge [->] node [right] {$h'$} (m-1-2);
		\end{tikzpicture}
	\end{center}
	\caption{Local rule}\label{fig:localrule}
\end{figure}
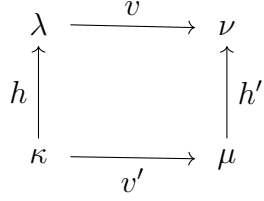

We say that this cell \Dfn{satisfies the local rules} if the following
two conditions are satisfied.
These two conditions are equivalent by Lemma~\ref{lemma:rev}.
\begin{equation*}
\tau^A_{B,C}(a,h,v)=(a,v',h')\qquad \tau^A_{C,B}(a,v',h')=(a,h,v)
\end{equation*}
where $a$ is the highest weight element of $A$.

A \Dfn{growth diagram} is then a diagram of cells in which each
cell satisfies the local rules.

\begin{defn} Let $B$ be a crystal and $w\in B$. Then the \Dfn{rectification} of $w$, $\rt(w)$, is the unique highest
	weight element in $B$ such that $w$ and $\rt(w)$ are in the
	same connected component of $B$.
\end{defn}

\begin{lemma} Let $B$ be a crystal and $w\in B$. Choose a crystal
	$A$ and $u\in A$ such that $u\otimes w\in A\otimes B$ is
	highest weight. Then $\sigma_{A,B}(u,w) = (\rt(w),u')$.
\end{lemma}

\begin{rem} The element $u'$ is also given by the Kashiwara
involution on the crystal $B_\infty$.
\end{rem}

\begin{proof} Put $\sigma_{A,B}(u,w) = (w',u')$.
Then $w'$ is highest weight
because $w'u'$ is highest weight. Also $w$ and $w'$ are in the same
component because $\sigma$ is a crystal morphism.
\end{proof}

In terms of diagrams, we have that Lemmas~\ref{lemma:addb} and~\ref{lemma:addt} say that these are equal and similarly for two squares stacked vertically.
\begin{equation}\label{eqn:sigma}
\begin{tikzpicture}
\matrix (m) [matrix of math nodes,row sep=3em,column sep=3em,minimum width=2em]
{
	\lambda & \nu & \sigma \\
	\kappa & \mu & \rho \\};
\path[-stealth]
(m-2-1) edge [->] node [left] {$h$} (m-1-1);
\path[-stealth]
(m-1-1) edge [->] node [above] {$u$} (m-1-2);
\path[-stealth]
(m-1-2) edge [->] node [above] {$v$} (m-1-3);
\path[-stealth]
(m-2-1) edge [->] node [below] {$u'$} (m-2-2);
\path[-stealth]
(m-2-2) edge [->] node [below] {$v'$} (m-2-3);
\path[-stealth]
(m-2-3) edge [->] node [right] {$h''$} (m-1-3);
\path[-stealth]	
(m-2-2)	edge [->] node [right] {$h'$} (m-1-2);
\end{tikzpicture}
\begin{tikzpicture}
\matrix (m) [matrix of math nodes,row sep=3em,column sep=3em,minimum width=2em]
{
	\lambda &  & \sigma \\
	\kappa &  & \rho \\};
\path[-stealth]
(m-2-1) edge [->] node [left] {$h$} (m-1-1);
\path[-stealth]
(m-1-1) edge [->] node [above] {$u\otimes v$} (m-1-3);
\path[-stealth]
(m-2-1) edge [->] node [below] {$u'\otimes v'$} (m-2-3);
\path[-stealth]
(m-2-3) edge [->] node [right] {$h''$} (m-1-3);
\end{tikzpicture}
\end{equation}

This shows that rectification is given in terms of local rules by a rectangular growth diagram. This generalises the construction
of rectification of skew tableaux using Fomin growth diagrams
given in \cite[Chapter~7: Appendix~1 Figure~A1-11]{Stanley1999}.

\subsection{Evacuation}
In this section we give the growth diagram description of evacuation.
We show that the result of applying $s_{1\,r}$ to a highest weight
word, as in \eqref{eqn:word}, can be read off a triangular diagram.
The triangular diagram for the case $r=3$ is shown in Figure~\ref{fig:tri}.

Define \Dfn{evacuation} to be the action of $s_{1,\,r}$.
The growth diagram for evacuation is triangular.

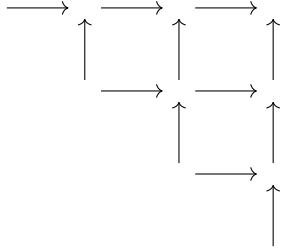
\begin{figure}
\begin{center}
	\begin{tikzpicture}
	\matrix(c) [matrix of math nodes,row sep=2em,column sep=2em,minimum width=1em]
	{
		 {}&  {}&  {}& {}&\\
		 {}&  {}&  {}& {}&\\
		 {}&  {}&  {}& {}&\\
		 {}&  {}&  {}& {}&\\};
    \path[->] (c-1-1) edge (c-1-2);
    \path[->] (c-1-2) edge (c-1-3);
    \path[->] (c-1-3) edge (c-1-4);
    \path[->] (c-2-2) edge (c-2-3);
    \path[->] (c-2-3) edge (c-2-4);
    \path[->] (c-3-3) edge (c-3-4);
    \path[->] (c-2-2) edge (c-1-2);
    \path[->] (c-2-3) edge (c-1-3);
    \path[->] (c-2-4) edge (c-1-4);
    \path[->] (c-3-3) edge (c-2-3);
    \path[->] (c-3-4) edge (c-2-4);
    \path[->] (c-4-4) edge (c-3-4);
	\end{tikzpicture}
\end{center}
\caption{A triangular diagram}\label{fig:tri}
\end{figure}

\begin{defn} A \Dfn{triangular growth diagram} is a triangular diagram
with vertices labelled by dominant weights and edges labelled by elements
of crystals such that the labels of each cell satisfy the local rules
as represented in Figure~\ref{fig:localrule}. We also require that the
first vertex on each row is labelled by the zero weight.
\end{defn}
These conditions imply that a triangular growth diagram can be reconstructed from the top edge viewed as a highest weight word.
\begin{prop}\label{prop:evac} Given a triangular growth diagram
let $w$ be the top edge viewed as a highest weight word and let
 $w'$ be the right edge viewed as a highest weight word. Then
 $w'=s_{1\,r}(w)$.
\end{prop}

This generalises the construction of the Sch\"utzenberger involution (aka evacuation) of tableaux in \cite[Figure~A1-13]{Stanley1999}.

\begin{proof} The proof is by induction on $r$. The basis of the induction
is the case $r=2$.

The inductive step is based on the relation $s_{1,\,r+1} = \sigma_{1,\,r,\,r+1}\, s_{1,\,r}$ which is a special case of \eqref{eqn:rel}.
The inductive step is given by interpreting this in terms of growth diagrams. The inductive hypothesis gives the growth diagram for $s_{1,\,r}$ and \eqref{eqn:sigma} (with the squares stacked vertically) gives the growth diagram for $\sigma_{1,\,r,\,r+1}$.
\end{proof}

Note that it follows from the symmetry of the local rules that the
reflection of a triangular growth diagram is also a triangular growth diagram. This shows that $w'=s_{1\,r}(w)$ if and only if $w=s_{1\,r}(w')$
so evacuation is an involution.

\subsection{Promotion}
In this section we give the growth diagram description of promotion.
Define \Dfn{promotion} to be the action of $s_{1,\,r}\,s_{2,\,r}$.
The growth diagram for promotion consists of two rows.

The two row diagram for the case $r=3$ is shown in Figure~\ref{fig:pro}.
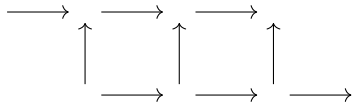
\begin{figure}
	\begin{center}
		\begin{tikzpicture}
		\matrix(c) [matrix of math nodes,row sep=2em,column sep=2em,minimum width=1em]
		{
			{}&  {}&  {}& {}& {}\\
			{}&  {}&  {}& {}& {}\\
			};
		\path[->] (c-1-1) edge (c-1-2);
		\path[->] (c-1-2) edge (c-1-3);
		\path[->] (c-1-3) edge (c-1-4);
		\path[->] (c-2-2) edge (c-2-3);
		\path[->] (c-2-3) edge (c-2-4);
		\path[->] (c-2-2) edge (c-1-2);
		\path[->] (c-2-3) edge (c-1-3);
		\path[->] (c-2-4) edge (c-1-4);
		\path[->] (c-2-4) edge (c-2-5);
		\end{tikzpicture}
	\end{center}
	\caption{A two row diagram}\label{fig:pro}
\end{figure}

\begin{defn} A \Dfn{two row growth diagram} is a two row diagram
	with vertices labelled by dominant weights and edges labelled by elements
	of crystals such that the labels of each cell satisfy the local rules
	as represented in Figure~\ref{fig:localrule}. We also require that the
	first vertex on each row is labelled by the zero weight.
\end{defn}
These conditions imply that a two row growth diagram can be reconstructed from the top edge viewed as a highest weight word.
\begin{prop}\label{prop:evac} Given a two row growth diagram
	let $w$ be the top edge viewed as a highest weight word and let
	$w'$ be the bottom edge viewed as a highest weight word. Then
	$w'=s_{1,\,r}\,s_{2,\,r}(w)$.
\end{prop}

This generalises the construction of promotion on tableaux
given in~\cite[Figure~6]{Rubey2011}.

\subsection{Cylindrical diagrams}
Cylindrical growth diagrams for $\GL(N)$ are used to describe the
action of various elements of $\fC_r$ on standard tableaux of size $r$.
This construction is given in~\cite[\S~6]{Speyer2014},~\cite{White2015}
and~\cite{Akhmejanov2017}.
In this section we use local rules to generalise
this construction to any crystal.
\begin{defn}
	Let $\mathbb{I}\subset \bZ^2$ be the subset consisting of pairs $(i,j)$ such that $0\le j-i\le r$. A \Dfn{\cgd} of shape $\lambda$ is a function $\gamma$ from $\mathbb{I}$
	to highest weights such that
	\begin{itemize}
		\item $\gamma(i,j)=0$ if $j-i=0$
		\item $\gamma(i,j)=\lambda$ if $j-i=r$
		\item Every unit square satisfies the local rule.
	\end{itemize}
\end{defn}

The \Dfn{type} is the sequence of dominant weights
$\lambda^{(i)} = \gamma(i,i+1)$. The set of growth diagrams
of type $\lambda^{(i)}$ and shape $\lambda$ corresponds to 
highest weight elements of
$C(\lambda^{(1)})\otimes C(\lambda^{(2)})\otimes \dotsb \otimes C(\lambda^{(r)})$ of weight $\lambda$. This follows because all horizontal edges in the same column are labelled by the same crystal and all
vertical edges in the same row are labelled by the same crystal.
Furthermore these two sequences agree.

\begin{ex} This is a \cgd. The rows are standard tableaux.
	
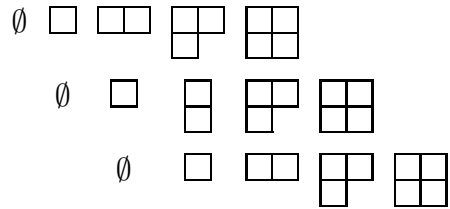
\begin{center}
		\begin{tikzpicture}
		\ytableausetup{smalltableaux}
		\matrix(c)[matrix of math nodes] {
\emptyset&\ydiagram{1} & \ydiagram{2} & \ydiagram{2,1} & \ydiagram{2,2} \\
&\emptyset&\ydiagram{1} & \ydiagram{1,1} & \ydiagram{2,1} & \ydiagram{2,2} \\
&&\emptyset&\ydiagram{1} & \ydiagram{2} & \ydiagram{2,1} & \ydiagram{2,2} \\		
		};
		\end{tikzpicture}
	\end{center}	
	\end{ex}

\begin{defn} A \Dfn{path} through $\mathbb{I}$ is a sequence $(i_0,j_0), (i_1,j_1), \dotsc , (i_r,j_r)$ of elements of $\mathbb{I}$ such that $i_0=j_0$ and, for $0\le k< r$, the difference $(i_{k+1},j_{k+1})-(i_k,j_k)$ is either $(-1,0)$ or $(0,1)$.
\end{defn}

Given a \cgd, $\gamma$ and a path we can restrict $\gamma$ to the path by taking the sequence $\gamma(i_0,j_0), \gamma(i_1,j_1), \dotsc , \gamma(i_r,j_r)$ to get a highest weight vector. For a fixed path this is a bijection between \cgd{s} and paths. The inverse map is giving by extending using the boundary conditions and the local rules.

Each path defines an operation on highest weight words. These
operations on a highest weight word $w$ are constructed by taking
the cylindrical growth diagram with $w$ on the top edge.
Then the result of applying the operation to $w$ is given by
reading the path as a highest weight word. For example,
the horizontal paths give the powers of the promotion operator.

The following is a generalisation of~\cite[Proposition~3.31]{White2015}.
This is an extension of the description of wall crossing
in~\cite[\S~6.1]{Speyer2014}. This is illustrated in~\cite[Figure~4]{White2015}.

\begin{defn}	
Define an operator $[p:q]$ on {\cgd}s of length $r$ and shape $\lambda$.
The action of $[p:q]$ on $\gamma$ satisfies
\begin{equation}
([p:q]\gamma)(i,j)=\begin{cases}
\gamma(i,j)&\text{if $\os{i:j}\cap\os{p:q}=\emptyset$} \\
\gamma(i,j)&\text{if $\os{p:q}\subset\os{i:j}$} \\
\gamma(i',j')&\text{if $\os{i:j}\subset\os{p:q}$}
\end{cases}
\end{equation}
where $i'=p+q-j$ and $j'=p+q-i$.
These conditions specify $([p:q]\gamma)$ on the path $(p,j)$, $j\ge p$ and so determine $([p:q]\gamma)$.
\end{defn}

\begin{thm} The map $s_{p,q}\mapsto [p:q]$ extends to an action of $\fC_r$ on highest weight words of length $r$ and shape $\lambda$.
\end{thm}

\begin{proof} This follows from Proposition~\ref{prop:evac} and
the identity \eqref{eqn:gen}.
\end{proof}

\begin{ex}\label{ex:sp} This is an example of a \cgd{} for a symplectic group.
	The rows are oscillating tableaux.
	
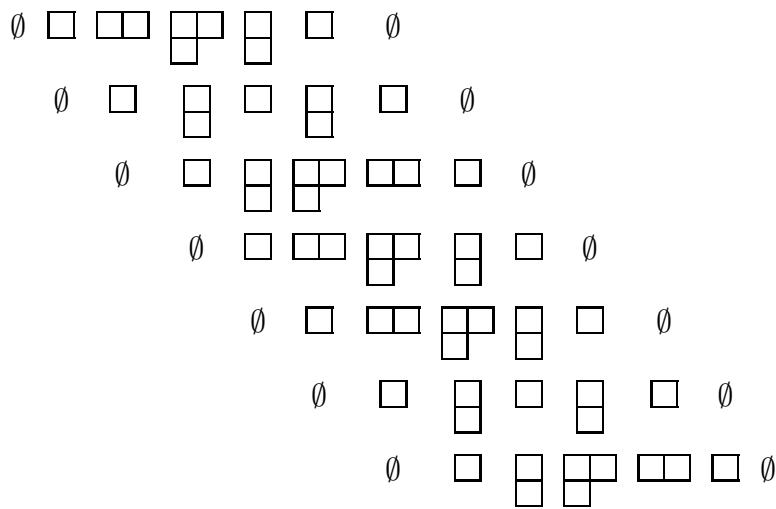
\begin{center}
		\begin{tikzpicture}
		\ytableausetup{smalltableaux}
		\matrix(c)[matrix of math nodes]
		{ \emptyset & \ydiagram{1} & \ydiagram{2} & \ydiagram{2,1}%
			& \ydiagram{1,1} & \ydiagram{1} & \emptyset \\
			& \emptyset & \ydiagram{1} & \ydiagram{1,1} & \ydiagram{1}%
			& \ydiagram{1,1} & \ydiagram{1} & \emptyset \\
			& & \emptyset & \ydiagram{1} & \ydiagram{1,1} & \ydiagram{2,1}%
			& \ydiagram{2} & \ydiagram{1} & \emptyset \\
			& & &   \emptyset & \ydiagram{1} & \ydiagram{2} & \ydiagram{2,1}%
			& \ydiagram{1,1} & \ydiagram{1} & \emptyset \\
			&&&&\emptyset & \ydiagram{1} & \ydiagram{2} & \ydiagram{2,1}%
			& \ydiagram{1,1} & \ydiagram{1} & \emptyset \\
			&&&&& \emptyset & \ydiagram{1} & \ydiagram{1,1} & \ydiagram{1}%
			& \ydiagram{1,1} & \ydiagram{1} & \emptyset \\
			& &&&&& \emptyset & \ydiagram{1} & \ydiagram{1,1} & \ydiagram{2,1}%
			& \ydiagram{2} & \ydiagram{1} & \emptyset \\
		};
		\end{tikzpicture}
	\end{center}
	If we start with the oscillating tableau on the first row
	
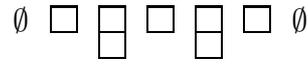
\begin{center}
		\begin{tikzpicture}
		\ytableausetup{smalltableaux}
		\matrix(c)[matrix of math nodes]
		{ \emptyset & \ydiagram{1} & \ydiagram{1,1} & \ydiagram{1}%
			& \ydiagram{1,1} & \ydiagram{1} & \emptyset \\
		};
		\end{tikzpicture}
	\end{center}
	Promotion is given by the second row
	
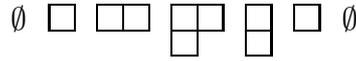
\begin{center}
		\begin{tikzpicture}
		\ytableausetup{smalltableaux}
		\matrix(c)[matrix of math nodes]
		{ \emptyset & \ydiagram{1} & \ydiagram{2} & \ydiagram{2,1}%
			& \ydiagram{1,1} & \ydiagram{1} & \emptyset \\
		};
		\end{tikzpicture}
	\end{center}
	Evacuation is given by the column with head $\emptyset$
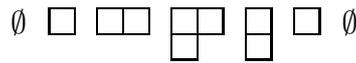
	\begin{center}
		\begin{tikzpicture}
		\ytableausetup{smalltableaux}
		\matrix(c)[matrix of math nodes]
		{ \emptyset & \ydiagram{1} & \ydiagram{2} & \ydiagram{2,1}%
			& \ydiagram{1,1} & \ydiagram{1} & \emptyset \\
		};
		\end{tikzpicture}
	\end{center}
\end{ex}
\begin{ex}	In this example we give the action of
$s_{3,6}$ on the oscillating tableau in Example~\ref{ex:sp}. We start by filling in the
	third row or the sixth column. Then we complete the diagram using the local rules. This gives:
	
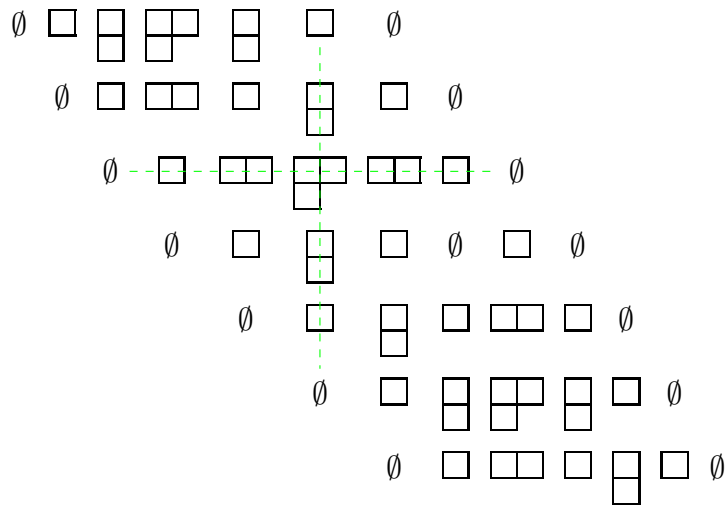
\begin{center}
		\begin{tikzpicture}
		\ytableausetup{smalltableaux}
		\matrix(c)[matrix of math nodes]
		{ \emptyset & \ydiagram{1} & \ydiagram{1,1} & \ydiagram{2,1}%
			& \ydiagram{1,1} & \ydiagram{1} & \emptyset \\
			& \emptyset & \ydiagram{1} & \ydiagram{2} & \ydiagram{1}%
			& \ydiagram{1,1} & \ydiagram{1} & \emptyset \\
			& & \emptyset & \ydiagram{1} & \ydiagram{2} & \ydiagram{2,1}%
			& \ydiagram{2} & \ydiagram{1} & \emptyset \\
			& & &   \emptyset & \ydiagram{1} & \ydiagram{1,1} & \ydiagram{1}%
			& \emptyset & \ydiagram{1} & \emptyset \\
			&&&&\emptyset & \ydiagram{1} & \ydiagram{1,1} & \ydiagram{1}%
			& \ydiagram{2} & \ydiagram{1} & \emptyset \\
			&&&&& \emptyset & \ydiagram{1} & \ydiagram{1,1} & \ydiagram{2,1}%
			& \ydiagram{1,1} & \ydiagram{1} & \emptyset \\
			& &&&&& \emptyset & \ydiagram{1} & \ydiagram{2} & \ydiagram{1}%
			& \ydiagram{1,1} & \ydiagram{1} & \emptyset \\
		};
		\draw[green,dashed] (c-3-3) -- (c-3-9);
		\draw[green,dashed] (c-1-6) -- (c-6-6);
		\end{tikzpicture}
	\end{center}
Then we take the oscillating tableaux on the first row.
\end{ex}

\printbibliography
 
\end{document}